\documentclass[12pt]{amsart}
\usepackage{amsfonts, amsmath, amsthm}

\usepackage{graphicx}
\usepackage{epstopdf}
\usepackage{psfrag}
\usepackage{color}

\newtheorem{pro}{Proposition}[section]
\newtheorem{thm}[pro]{Theorem}
\newtheorem{lem}[pro]{Lemma}

\newtheorem{cor}[pro]{Corollary}

\newtheorem*{qthm}{Theorem}

\theoremstyle{definition}
\newtheorem{dfn}[pro]{Definition}

\theoremstyle{remark}

\newcommand{\bdy}{\partial}

\title{Locally Helical Surfaces have bounded twisting}
\date{\today}
\address{Pitzer College}
\email{bachman@pitzer.edu}
\author{David Bachman}
\begin{document}

\address{Quest University}
\email{rdt@questu.ca}
\author{Ryan Derby-Talbot}

\address{DePaul University}
\email{esedgwick@cdm.depaul.edu}
\author{Eric Sedgwick}

\begin{abstract}
A topologically minimal surface may be isotoped into a normal form with respect to a fixed triangulation. If the intersection with each tetrahedron is simply connected, then the pieces of this normal form are triangles, quadrilaterals, and helicoids. Helical pieces can have any number of positive or negative twists. We show here that the net twisting of the helical pieces of any such surface in a given triangulated 3-manifold is bounded. 
\end{abstract}

\maketitle

\section{Introduction}
\label{s:intro}

In \cite{TopIndexI}, the first author introduced the notion of a {\it topologically minimal} surface, as a generalization of incompressible \cite{haken:68}, strongly irreducible \cite{cg:87}, and critical \cite{crit} surfaces. Such surfaces have a well-defined {\it index}, where incompressible, strongly irreducible, and critical surfaces have indices 0, 1, and 2, respectively. 

The term ``topologically minimal" was chosen because in many ways, such surfaces behave like geometrically minimal surfaces, i.e.~surfaces that represent critical points for the area function. Similarities between the two types of surfaces are made explicit in e.g.~\cite{TopIndexI} and \cite{TopMinNormalII}, and one of the goals of the present paper to present further similarities. 

A useful fact about topologically minimal surfaces is that they can be isotoped into a  standard normal form with respect to a triangulation. This was first done by Kneser \cite{kneser:29} and Haken \cite{haken:61} in the index 0 case, Rubinstein \cite{rubinstein:93} and Stocking \cite{stocking:96} for closed index 1 surfaces,  and \cite{Index1Normal} for index 1 surfaces with boundary. The general case of arbitrary index is addressed by the first author in \cite{TopMinNormalI}, \cite{TopMinNormalII}, and \cite{TopMinNormalIII}. The following theorem summarizes these results:

\begin{thm}
\label{Global2Local}
Let $M$ be a compact, orientable, irreducible, triangulated 3-manifold with incompressible boundary. Then for each $n$ there exists a finite, constructible set of surfaces in each tetrahedron of $M$ from which one can build any index $n$ topologically minimal surface in $M$ (up to isotopy). 
\end{thm}


The pieces from which index $n$ surfaces can be built by Theorem \ref{Global2Local} can be quite complicated. However, in \cite{TopMinNormalII} the first author gives a relatively simple characterization of those components that are simply connected: such pieces are either triangles or helicoids\footnote{We regard quadrilaterals as untwisted helicoids.} (see Figure \ref{f:16-gon}). We say any surface built entirely from such pieces is {\it locally helical.}

\begin{figure}
\[\includegraphics[width=3 in]{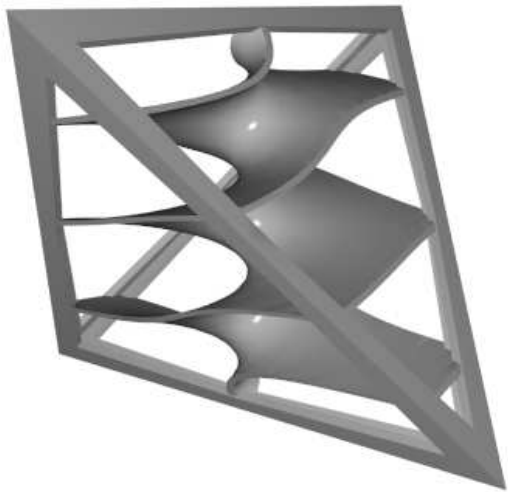}\]
\caption{A helicoid whose boundary has length 16. Note that it meets one pair of opposite edges in single points, a second pair in three points, and a third pair in four points. The twisting of this helicoid is 3.}
\label{f:16-gon}
\end{figure} 

Helical pieces are classified by their {\it axis} (see Section \ref{s:axis}) and {\it twisting.} If $H_*$ is a helicoid then the number of normal arcs comprising $\bdy H_*$ is $4(n+1)$, for some $n$. The {\it twisting} of $H_*$, denoted $t(H_*)$, is  the number $\pm n$, where the sign is determined by the handedness of the helicoid and the orientation of the manifold (see Definition~\ref{def:twisting}). If $H$ is a locally helical surface in a triangulated 3-manifold $M$, then the {\it net twisting} of $H$ is the sum of the twisting of all of its helical pieces (see Definition \ref{d:NetTwisting} for a more precise definition). The {\it total absolute twisting} is the sum of the absolute values of the twisting of its helical pieces. Note that if a surface has bounded total absolute twisting, then each helical piece has a bounded number of twists. If, on the other hand, the net twisting is bounded then there may be helical pieces with an arbitrarily large number of, say, positive twists, as long as there are also pieces with large numbers of negative twists.

The results of \cite{TopMinNormalI} and \cite{TopMinNormalII}, taken together, imply the following:

\begin{thm}
\label{t:BoundedLength}
Any topologically minimal surface with index $n$ that is isotopic to a locally helical surface is isotopic to one with total absolute twisting at most $n$. 
\end{thm}


The results mentioned above give a direct generalization of Haken's normalization of incompressible surfaces \cite{haken:68}. To see this, first note that by definition, an incompressible surface is index 0. By Theorem \ref{Global2Local} such a surface can be isotoped to be locally topologically minimal. By incompressibility, we may assume that in this position it is locally simply connected. Finally, by Theorem \ref{t:BoundedLength} we conclude that the total absolute twisting must be 0, which means that it is a collection of triangles and quadrilaterals.

For higher index locally helical surfaces, the situation may be more complicated, as there may be helicoids distributed across tetrahedra in $M$. The main result of this paper is the following theorem, which says that the total absolute twisting outside of some prescribed set of tetrahedra $\Delta$ constrains the net twisting inside $\Delta$.

\begin{thm}
\label{t:mainthm}
Let $M$ be a closed, oriented, triangulated 3-manifold, and let $\Delta$ be a set of tetrahedra in the triangulation of $M$. Let $H$ be a locally helical surface in $M$ such that the total absolute twisting  of $H -\Delta$ is at most $n$. Then the net twisting of $H \cap \Delta$ is bounded, where the bound depends only on $M$ and $n$.
\end{thm}

Three corollaries of this theorem are worth noting: where $\Delta$ is a single tetrahedron of $M$, where $\Delta$ is exactly two tetrahedra, and where $\Delta$ is the set of all tetrahedra in $M$.

\begin{cor}
Let $M$ be a closed, oriented, triangulated 3-manifold, and let $\Delta$ be a tetrahedron of the triangulation. Let $H$ be a locally helical surface in $M$ such that the total absolute twisting of $H -\Delta$ is at most $n$. Then the total absolute twisting of $H$ is bounded, where the bound depends only on $M$ and $n$.
\end{cor}

This follows since a bound on the net twisting of a surface in a single tetrahedron serves as a bound for its absolute value. 

\begin{cor}
Let $M$ be a closed, oriented, triangulated 3-manifold, and let $\Delta_1, \Delta_2$ be a pair of tetrahedra in the triangulation of $M$. Let $H$ be a locally helical surface in $M$ such that the total absolute twisting of $H -(\Delta_1 \cup \Delta_2)$ is at most $n$. Then $t(H \cap \Delta_1)=-t(H \cap \Delta_2)+m$, where $m$ is bounded by a function of  $M$ and $n$. 
\end{cor}

In other words, if, in a sequence of surfaces with bounded total absolute twisting outside of $\Delta_1 \cup \Delta_2$, the number of left-handed twists in $\Delta_1$ is growing, then so must be the number of right-handed twists in $\Delta_2$. This brings to light a striking resemblance between topologically minimal surfaces and geometrically minimal surfaces, as described by Colding and Minicozzi in the following theorem: 

\begin{qthm}[\cite{cmpnas}]
Any nonsimply connected embedded minimal planar domain without small necks can be obtained from gluing together two oppositely oriented double spiral staircases. Moreover, if for some point the curvature is large, then the separation between the sheets of the double spiral staircases is small. Note that because the two double spiral staircases are oppositely oriented, then one remains at the same level if one circles both axes.
\end{qthm}

The last corollary of Theorem \ref{t:mainthm} is when $\Delta$ is the set of all tetrahedra in $M$. In this case, our result makes no mention of total absolute twisting. 

\begin{cor}
Let $M$ be a closed, oriented, triangulated 3-manifold, and let $H$ be a locally helical surface in $M$. Then the net twisting of $H$ is bounded, where the bound depends only on $M$.
\end{cor}

In the next section we characterize normal curves by their {\it type}. In Section \ref{s:axis} we characterize helical disks by their {\it axis}. Finally, in Section \ref{s:compatibility} we define the {\em compatibility class} of a locally helical surface. Those familiar with normal surface theory will find several of these notions familiar. The layering of these definitions parses the set of locally helical surfaces in $(M;\Delta)$ more and more finely, imposing increasingly greater restrictions on how surfaces in the same class can intersect. Taken all at once, these characterizations produce a finite set of {\em consistency classes} for locally helical surfaces in $(M; \Delta$), which have just the properties needed to prove Theorem~\ref{t:mainthm}.

\section{The {\it type} of a normal curve on a tetrahedron.}

In this section we consider the combinatorics of normal loops on the boundary of a tetrahedron. For a basic reference on normal surface theory, we refer the reader to \cite{hass:98}. 

\begin{lem}
\label{l:Rotation}
Let $\sigma$ be a tetrahedron, and $\alpha$ a normal loop of length at least four on $\bdy \sigma$. Let $\phi$ denote a 180 degree rotation of $\sigma$ about a line connecting the midpoints of opposite edges of $\sigma$. Then $\alpha$ is normally parallel to a loop that is preserved by $\phi$. 
\end{lem}

\begin{proof}
To begin, we claim that a normal loop of length at least four meets each pair of opposite edges of $\bdy \sigma$ in the same number of points. One way to see this is by noting that the double cover of $\bdy \sigma$, branched over the vertex set, is a torus (see Figure~\ref{f:Cover_of_a_tet}). Each edge of $\bdy \sigma$ lifts to an essential loop on the torus, and each pair of opposite edges lifts to two parallel loops. Now, as a loop $\alpha$ of length at least four on $\bdy \sigma$ also lifts to two essential loops on the torus, it must be the case that $\alpha$ intersects opposite edges of $\bdy \sigma$ in an equal number of points. 

\begin{figure}
\[\includegraphics[width=5 in]{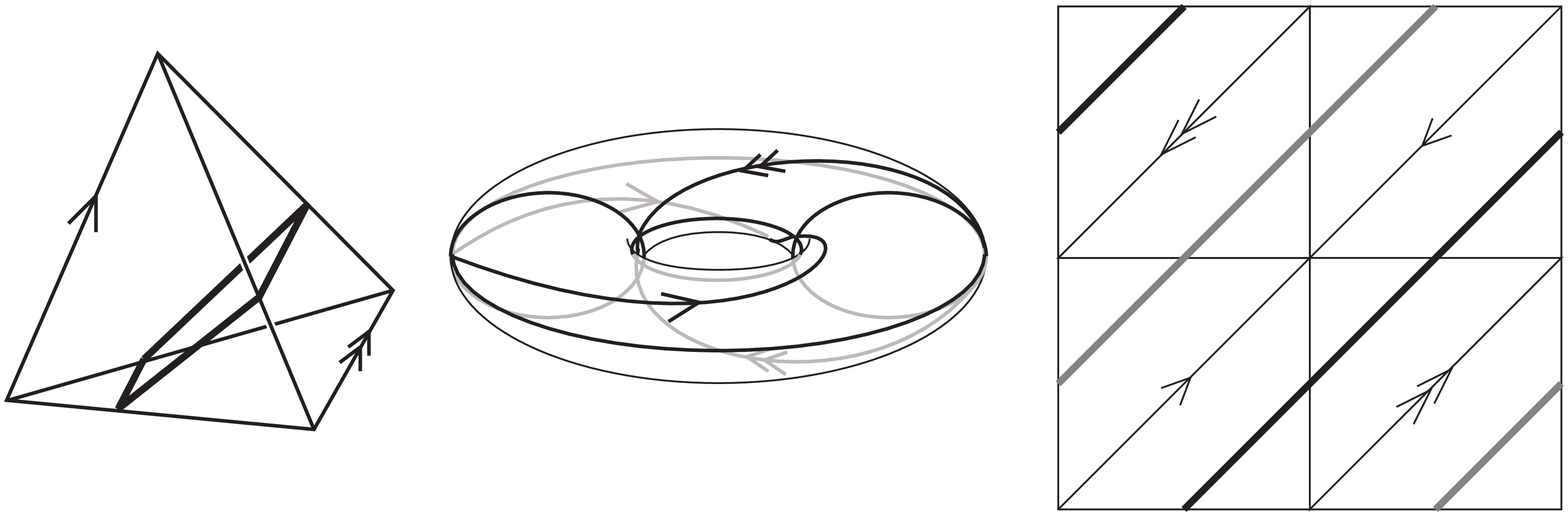}\]
\caption{The torus as a double branched cover of the boundary of a tetrahedron, and components of a lift of a length four curve in its unfolded version.}
\label{f:Cover_of_a_tet}
\end{figure}

Now note that the rotation $\phi$ preserves the two edges that its axis intersects, and swaps the other two pairs of opposite edges. Hence, both $\alpha$ and $\phi(\alpha)$ will meet each edge in the same number of points. As these numbers completely determine the intersection of $\alpha$ with each face of $\sigma$ (up to normal isotopy), the result follows. 
\end{proof}

This lemma gives us a way to classify normal curves on the boundary of a tetrahedron. Label the normal arc types on each face of a tetrahedron $\sigma$ as in Figure \ref{f:ArcTypes}. These labels are arranged so as to be preserved by 180 degree rotations about axes that connect the midpoints of opposite edges. Any normal loop $\alpha$ of length at least four on $\bdy \sigma$ meets each face in a collection of normal arcs. By Lemma \ref{l:Rotation}, the number of these arcs that are parallel to an arc with one label in one face will be the same as the number that are parallel to an arc with the same label in any other face. Hence, if we fix one face $\delta$ of $\sigma$ and let  $a(\alpha)$, $b(\alpha)$ and $c(\alpha)$ be the number of arcs of $\alpha \cap \delta$ parallel to the labelled arcs $a$, $b$, and $c$ of the figure, then these three functions will be independent of the choice of $\delta$. 

\begin{figure}
\psfrag{a}{$a$}
\psfrag{b}{$b$}
\psfrag{c}{$c$}
\[\includegraphics[width=3in]{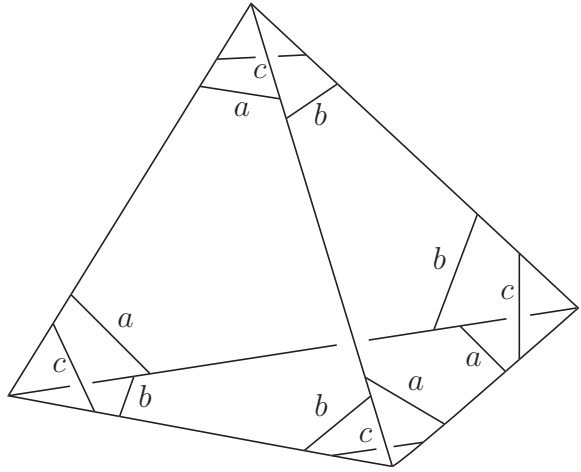}\]
\caption{Labeling the normal arc types on the boundary of a tetrahedron, $\sigma$.}
\label{f:ArcTypes}
\end{figure}

Note furthermore that for any loop $\alpha$ of length at least four, at least one of the three numbers $a(\alpha)$, $b(\alpha)$ or $c(\alpha)$ will be zero (otherwise $\alpha$ would have length three components). This motivates the following definition. 

\begin{dfn}
\label{d:type}
Let $\sigma$ be a tetrahedron with labeled normal arc types as in Figure~\ref{f:ArcTypes}, and let $\alpha$ be a normal loop on $\bdy \sigma$ of length at least four. We say $\alpha$ is {\it type $a$}, if $a(\alpha)=0$. Define {\it type $b$} and {\it type $c$} similarly. 
\end{dfn}

Note that normal loops of length exactly four will be of two types. The notion of type constrains how two normal curves can intersect on the boundary of a tetrahedron, as seen in the following two lemmas. 

\begin{lem}
\label{l:NoLengthThree}
Let $\alpha$ and $\beta$ be normal loops of length at least four on $\bdy \sigma$ of the same type. Let $\alpha+\beta$ be the normal loop(s) obtained by resolving all intersection points. Then $\alpha + \beta$ does not contain any components of length three.
\end{lem}

\begin{proof}
Suppose $\alpha$ and $\beta$ are type $a$. Then $a(\alpha)=a(\beta)=0$. As $a(\alpha+\beta)=a(\alpha)+a(\beta)$ for any two normal loops, we conclude $a(\alpha+\beta)=0$. Thus, there is a missing arc type around each vertex of $\sigma$ (see Figure \ref{f:ArcTypes}). We conclude $\alpha+\beta$ does not have any components of length three.
\end{proof}

\begin{dfn}
Let $\alpha_0$ and $\beta_0$ be normal arcs in an oriented triangle $\delta$. Then $\alpha_0$ and $\beta_0$ can be isotoped, keeping their boundaries fixed, so that they intersect transversely in at most one point.  We define the {\it (normal) sign} of the point $\alpha_0 \cap \beta_0$, if it exists, as follows. Orient $\alpha_0$ and $\beta_0$ so that the ordering $(\alpha_0,\beta_0)$ agrees with the orientation of $\delta$. There are now two possibilities. If the regular exchange at $\alpha_0 \cap \beta_0$  attaches the tail of $\alpha_0$ to the tip of $\beta_0$ then we say intersection point $\alpha_0 \cap \beta_0$ is {\it positive.} Otherwise we say it is {\it negative} (see Figure \ref{f:NormalSwitches}). 
\end{dfn}

Note that with a fixed orientation on $\delta$, the sign of $\alpha_0 \cap \beta_0$ is opposite the sign of $\beta_0 \cap \alpha_0$.

\begin{figure}
\psfrag{a}{$\alpha_0$}
\psfrag{b}{$\beta_0$}
\psfrag{R}{Regular Exchange}
\psfrag{+}{$\alpha_0 \cap \beta_0$ positive}
\psfrag{-}{$\alpha_0 \cap \beta_0$ negative}
\[\includegraphics[width=4 in]{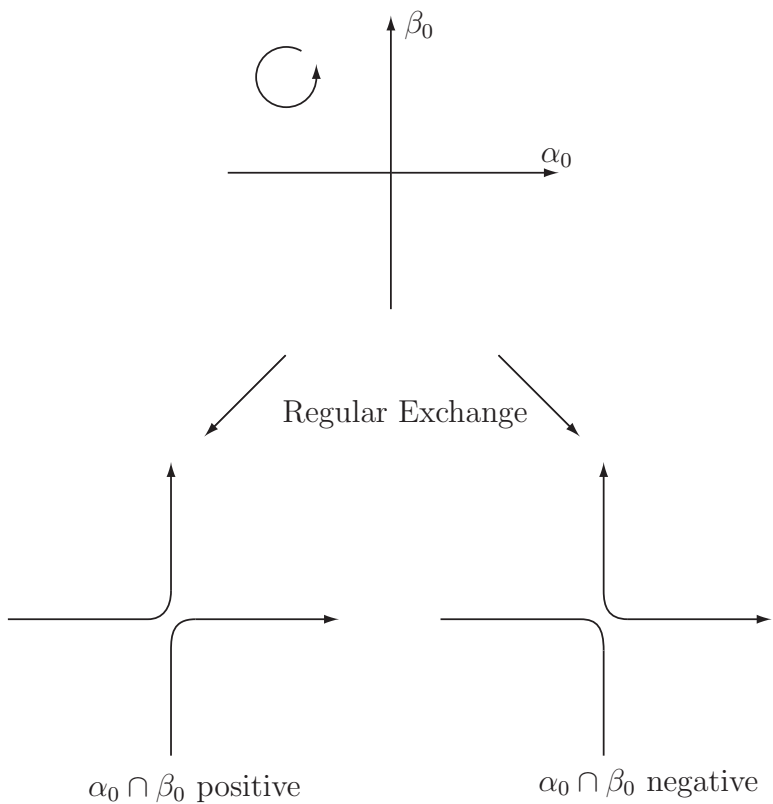}\]
\caption{The sign of $\alpha_0 \cap \beta_0$, as determined by the regular exchange.}
\label{f:NormalSwitches}
\end{figure}

\begin{lem}
\label{l:SameSign}
Let $\alpha$ and $\beta$ be collections of normal loops on $\bdy \sigma$ whose non-length three components are all of the same type, that have been normally isotoped to intersect minimally. Then each point of $\alpha \cap \beta$ has the same sign.
\end{lem}

\begin{proof}
If either $\alpha$ or $\beta$ contains a component of length three, then it will be disjoint from the other collection. Thus, we may assume that all points of $\alpha \cap \beta$ lie on loops of length at least four. We will call such loops {\it long}. The long loops of $\alpha$ will all be parallel, as will the long loops of $\beta$. Thus, if there are any intersection points at all, then no long loop of $\alpha$ can be parallel to a long loop of $\beta$. 

By way of contradiction, we now assume that two points of $\alpha \cap \beta$ are of opposite sign. We claim that then there is a subarc of $\alpha$ or $\beta$ that connects two points of $\alpha \cap \beta$ of opposite sign. If not, then we may choose a component $\alpha_+$ of $\alpha$ with only positive intersection points, and a component $\beta_-$ of $\beta$ with only negative intersection points. However, it then follows that $\alpha_+$ is disjoint from $\beta_-$, which cannot happen for two non-parallel long loops. We proceed, then, without loss of generality assuming there is a subarc of $\beta$ that connects points of opposite sign. It follows that there is such a subarc, $\beta_0$, which does not meet $\alpha$ in its interior. 

\begin{figure}
\psfrag{a}{$\alpha'$}
\psfrag{A}{$\alpha''$}
\psfrag{b}{$\beta_0$}
\[\includegraphics[width=4.5 in]{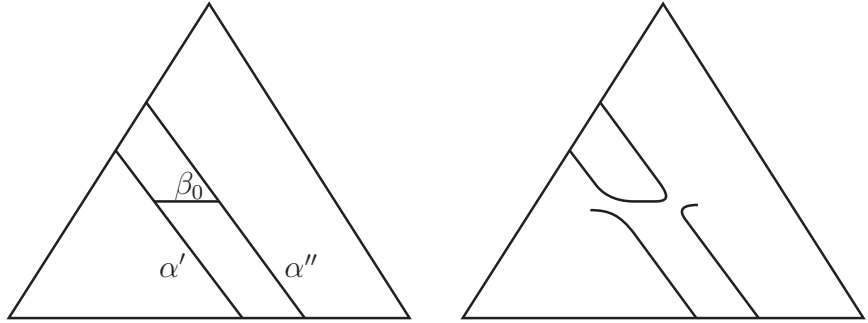}\]
\caption{Resolving intersections of opposite signs produces a non-normal arc.}
\label{f:Arcs_in_face}
\end{figure}

There are now two cases. Suppose first that the points of $\bdy \beta_0$ lie on different components $\alpha'$ and $\alpha''$ and of $\alpha$. As all long components of $\alpha$ are normally parallel, $\alpha'$ and $\alpha''$ cobound an annulus of $\bdy \sigma$, with $\beta_0$ a spanning arc. By making this annulus thin, we may assume that $\beta_0$ lies in a face of $\sigma$. However, resolving the two intersections at each end of $\beta_0$ then produces a non-normal arc. (See Figure \ref{f:Arcs_in_face}.)

\begin{figure}
\psfrag{a}{(a)}
\psfrag{b}{(b)}
\psfrag{c}{$\alpha'$}
\psfrag{d}{(c)}

\[\includegraphics[width=4.5 in]{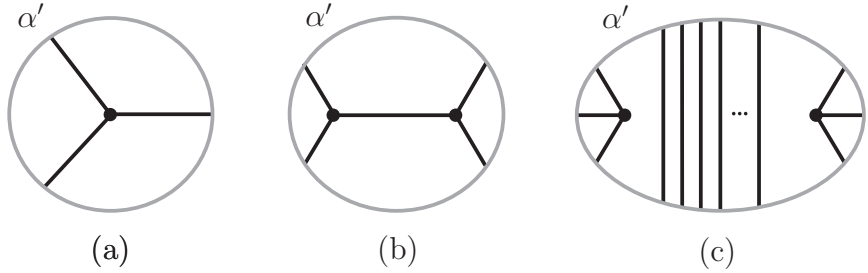}\]
\caption{The three possibilities for one hemisphere of $\bdy \sigma$, bounded by $\alpha'$. The black points and edges indicate vertices and suburbs of edges of $\bdy \sigma$, respectively.}
\label{f:tetCurve}
\end{figure}

The second case is when the points of $\bdy \beta_0$ lie on the same component $\alpha'$ of $\alpha$. Note that $\alpha'$ is a loop that divides $\bdy \sigma$ into two hemispheres, each intersecting the boundary of the tetrahedron in one of the three ways as seen in Figure~\ref{f:tetCurve}. The loop $\alpha'$ cannot be as depicted in Figure \ref{f:tetCurve}(a), where one of these hemispheres contains a single vertex of $\bdy \sigma$, since it is long. Thus, we may assume both hemispheres contain two vertices of $\bdy \sigma$. Let $D$ be the hemisphere that contains $\beta_0$. Note that $\beta_0$ then divides $D$ into two subdisks, and by the minimality of $|\alpha \cap \beta|$, each such subdisk will contain a vertex of $\sigma$. Resolving the intersections at each end of $\beta_0$ then produces a vertex linking loop. (See Figure \ref{f:Resolving_length_3_curve}). We will leave it as an exercise for the reader that such a loop will then persist after all further resolutions, producing a length three normal loop.  By Lemma \ref{l:NoLengthThree}, it follows that the long loops of $\alpha$ and $\beta$ could not have been the same type. 
\end{proof}

\begin{figure}
\psfrag{a}{$\alpha'$}
\psfrag{b}{$\beta_0$}
\psfrag{l}{length 3 loop}
\psfrag{n}{non-normal arc}
   \centering
   \includegraphics[width=3in]{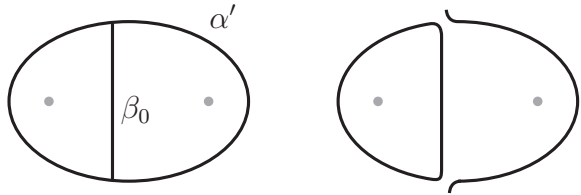} 
   \caption{Resolving the intersection points of opposite sign of $\alpha'$ and $\beta_0$ produces a length three curve or non-normal arc.}
\label{f:Resolving_length_3_curve}
\end{figure}

\section{Helicoids with the same {\it axis}.}
\label{s:axis}

\begin{dfn}
Let $H_*$ be a disk properly embedded in a tetrahedron, whose boundary is a normal loop. If $\bdy H_*$ meets some pair of opposite edges $e, e'$ in single points then we say $H_*$ is a {\it helicoid}, and $\{e,e'\}$ is an {\it axis} of $H_*$. 
\end{dfn}

Note that both quadrilaterals and octagons are helicoids with two axes, and all other helicoids have a unique axis. (See Figure~\ref{f:Quad_and_octagon}.) However, the boundary of each helicoid with axis $\{e,e'\}$ meets $e$ in a unique normal arc type as in Figure \ref{f:ArcTypes}.  

\begin{dfn}
\label{def:handedness}
Given a helicoid $H_*$ with axis $\{e,e'\}$ in a tetrahedron $\sigma$, there is an orientation-preserving simplicial homeomorphism from $\sigma$ to the tetrahedron pictured in Figure \ref{f:ArcTypes} (equipped with the standard orientation on $\mathbb R^3$), where $e$ and $e'$ are taken to the edges that meet arc types $a$ and $b$. We say $H_*$ is {\it right-handed with respect to $\{e,e'\}$} if $a(\partial H_*)=0$ and {\it left-handed with respect to $\{e,e'\}$} if $b(\partial H_*)=0$.
\end{dfn}

\begin{dfn}
\label{def:twisting}
Let $H_*$ be a helicoid with $4(n+1)$ normal arcs comprising $\bdy H_*$, and with axis $\{e,e'\}$ in a tetrahedron $\sigma$. We say the {\it twisting} of $H_*$, $t(H_*)$, is $+n$ if $H_*$ is right-handed with respect to $\{e,e'\}$ and $-n$ if it is left-handed with respect to $\{e,e'\}$.
\end{dfn}

\begin{figure}
\psfrag{e}{\tiny{$e$}}
\psfrag{p}{\tiny{$e'$}}
\psfrag{f}{\tiny{$f$}}
\psfrag{g}{\tiny{$f'$}}
   \centering
   \includegraphics[width=3in]{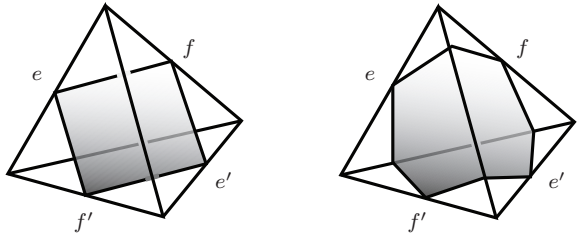} 
   \caption{Quadrilaterals and octagons are the only two locally helical surfaces with more than one choice of axis. Note here that the quadrilateral is left-handed and the octagon is right-handed with respect to $\{e, e'\}$, with the opposite being the case with respect to $\{f, f'\}$.}
\label{f:Quad_and_octagon}
\end{figure}

Note that an octagon can be regarded as having $+1$ or $-1$ twisting, depending on the choice of its axis. The handedness of the twisting of a quadrilateral is also dependent on a choice of axis, but in either case the value of the twisting is zero. Thus, a helical surface with no octagons has a well-defined net twisting. When there are octagons present, however, the net twisting will depend on choices of axes, motivating the following definition:

\begin{dfn}
\label{d:NetTwisting}
Let $M$ be a triangulated 3-manifold containing a locally helical surface $H$, and let $\Delta$ be a set of tetrahedra in the triangulation of $M$. We say the {\it net twisting of $H$ in $\Delta$ is bounded by $n$} if 
\[-n \le \sum \limits _{\sigma \in \Delta} t(H \cap \sigma) \le n\]
for all choices of axes of the components of $H \cap \sigma$, for each $\sigma \in \Delta$. 
\end{dfn}

\begin{dfn}
\label{d:eta}
Let $\sigma$ be an oriented tetrahedron. For any two normal curves $\alpha$ and $\beta$ on $\bdy \sigma$ in general position, let $\eta_\sigma(\alpha \cap \beta)$ denote the difference between the total number of positive and negative intersection points of $\alpha \cap \beta$ on the 2-simplices of $\bdy \sigma$.
\end{dfn}

\begin{lem}
\label{l:NumTurns}
Let $H_*$ and $G_*$ be helicoids with the same handedness with respect to the same choice of axis. Then \[\eta_\sigma(\bdy H_* \cap \bdy G_*)=2(t(H_*)-t(G_*)).\]
\end{lem}

\begin{proof}
As $H_*$ and $G_*$ are helicoids with the same handedness with respect to some choice of axis, it follows that their boundaries are loops of the same type.  It thus  follows immediately from Lemma~\ref{l:SameSign} that $\eta_{\sigma} (\bdy H_* \cap \bdy G_*) = \pm |\bdy H_* \cap \bdy G_*|$, where the sign is determined by the normal intersection sign of the intersection points. Without loss of generality, assume this sign is positive. Our goal is to show 
\[|\bdy H_* \cap \bdy G_*|=2(t(H_*)-t(G_*)).\]

$\bdy H_*$ is a loop on $\bdy \sigma$ dividing it into two hemispheres, where each hemisphere contains two of the vertices of $\sigma$. Let $v$ and $w$ be the vertices in one such hemisphere, and let $h$ be an arc in this hemisphere connecting them. Note that the arc $h$ can be chosen so that $\bdy H_*$ is normally parallel to a neighborhood of $h$. 

Similarly, $\bdy G_*$ is parallel to the boundary of a neighborhood of an arc $g$ connecting two vertices of $\sigma$. The arc $g$ may be chosen so that at least one of it's endpoints is distinct from the endpoints of $h$.

\begin{figure}
\[\includegraphics[width=5 in]{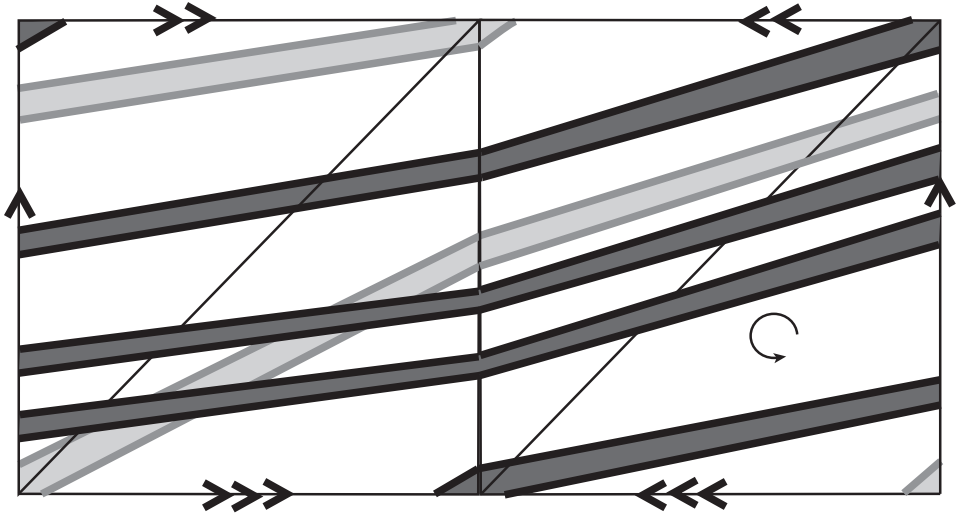}\]
\caption{$\bdy H_*$ is the black curve, with a neighborhood of $h$ being the dark gray band. $\bdy G_*$ is depicted in lighter gray. Here, $|h \cap g|=2$, $\eta_\sigma(\bdy H_* \cap \bdy G_*) =8$, $t(H_*)=6$, $t(G_*)=2$ and thus $t(H_*)-t(G_*)=4$.}
\label{f:UnfoldedCase1}
\end{figure}

There are now two cases. If both endpoints of $g$ are distinct from the endpoints of $h$ then the curves can be arranged as in Figure \ref{f:UnfoldedCase1}. Note that $\bdy H_* \cap \bdy G_*$ contains four intersection points for each crossing of $h$ and $g$. Furthermore, the difference in the twisting, $t(H_*)-t(G_*)$ is twice the number of crossings of $h$ and $g$. Thus, the desired equation holds. 

All intersection points depicted in the figure are positive, as is the twisting. Note that switching the orientation and keeping the ordering of the curves the same changes the sign of both the intersection points and the twisting. Alternatively, keeping the orientation fixed but changing the ordering of the curves will also change the sign of the intersection points, and reverse the order of the operands on the right side of the desired equation. Thus the equation still holds. 

\begin{figure}
\[\includegraphics[width=5 in]{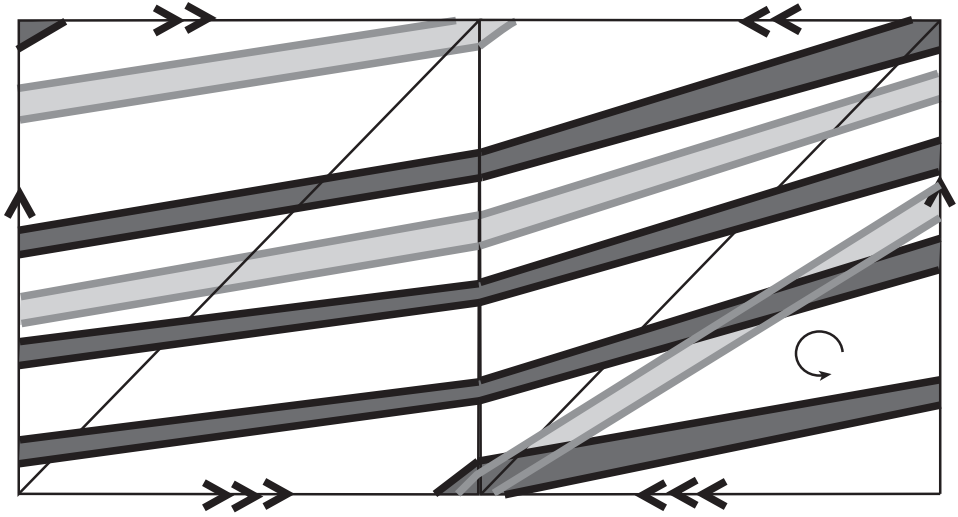}\]
\caption{In this case, $h$ and $g$ share an endpoint. Here, $|h \cap g|=2$, $\eta_\sigma(\bdy H_* \cap \bdy G_*) =6$, $t(H_*)=6$, $t(G_*)=3$ and thus $t(H_*)-t(G_*)=3$.}
\label{f:UnfoldedCase2}
\end{figure}

In the second case, $h$ and $g$ have an endpoint in common, as in Figure \ref{f:UnfoldedCase2}. In this case $|\bdy H_* \cap \bdy G_*|=4|h \cap g|-2$ and $t(H_*)-t(G_*)=2|h \cap g|-1$. Thus we still obtain the desired relationship between $\eta_\sigma(\bdy H_* \cap \bdy G_*)$ and $t(H_*)-t(G_*)$. 

\end{proof}

\section{{\em Compatibility classes} of surfaces}
\label{s:compatibility}

The results of this section extend previous results that restrict intersections of boundary curves realized by compatibility classes of surfaces, found e.g.~in \cite{JSJ}, \cite{js:98}, and \cite{Hatcher}. 

\begin{dfn}
Two surfaces in a triangulated 3-manifold are {\it compatible} if they meet the boundary of each tetrahedron in a collection of normal curves that can be normally isotoped to be disjoint\footnote{We are allowing pseudo-triangulations, i.e.~$M$ is realized as a collection of tetrahedra with face-pairings. Hence, for each 3-cell $\sigma$ in $M$ there is a map $\pi:\Sigma \to \sigma$, where $\Sigma$ is a 3-simplex. Here we consider two surfaces to be compatible if they meet $\bdy \sigma$ in curves whose  preimages can be isotoped to be disjoint on $\bdy \Sigma$.}.
\end{dfn}

Henceforth we will assume that if $\alpha_0$ and $\beta_0$ are contained in a 2-simplex $\delta \subset \bdy M$, then the orientation on $\delta$ is induced by the orientation on $M$. Hence, for such curves we may reference the sign of each point of $\alpha_0 \cap \beta_0$ without mention of the orientation of the 2-simplex that contains it. 

In the next lemma, we show that two compatible surfaces have a symmetric relationship between the signs of their normal intersections on the boundary of a subcomplex, $\Delta$.

\begin{lem}
\label{l:NetZero}
Let $M$ be a closed, oriented, triangulated 3-manifold. Let $\Delta$ be a set of tetrahedra in the triangulation of $M$. Suppose $A$ and $B$ are two  locally helical surfaces in $M$ that are compatible outside $\Delta$. Let $\bdy _\Delta A=\bdy (A \cap \Delta)$ and $\bdy _\Delta B=\bdy (B \cap \Delta)$. Suppose $A$ and $B$ have been normally isotoped so that $|\bdy _\Delta A \cap \bdy _\Delta B|$ is minimal. Then the number of points of $\bdy _\Delta A \cap \bdy _\Delta B$ with positive normal sign equals the number of points with negative normal sign.
\end{lem}

\begin{proof}
Consider a tetrahedron $\sigma$ of $M$ that is not in $\Delta$. Let $\alpha$ denote a component of $A \cap \bdy \sigma$, and $\beta$ a component of $B \cap \bdy \sigma$. Orient each 2-simplex of $\bdy \sigma$ by the induced orientation from $\sigma$, so that each point of $\alpha \cap \beta$ has a well-defined sign. Since $A$ and $B$ intersect minimally, we may assume each normal arc of $\alpha$ and $\beta$ is a straight line segment. Recall from Definition \ref{d:eta} that $\eta_\sigma(\alpha \cap \beta)$ denotes the difference between the total number of positive and negative intersection points of $\alpha \cap \beta$ on the 2-simplices of $\bdy \sigma$.


As $A$ and $B$ are compatible, there is an isotopy from $\alpha$ to a normal loop $\alpha'$, also consisting of straight normal arcs, in $\bdy \sigma$ that is disjoint from $\beta$. We can choose such an isotopy, $\alpha_t$, so that for all $t$, each normal arc of $\alpha_t$ is a straight line segment and $\alpha_t \cap \beta$ contains at most one point of the 1-skeleton. Let $\{t_i\}$ denote the critical values of $\alpha_t \cap \beta$, i.e.~the values of $t$ such that $\alpha_t$ and $\beta$ do not intersect transversely on $\partial \sigma$. It follows that for each $i$, $\alpha_{t_i} \cap \beta$ includes a point of the 1-skeleton.

\begin{figure}
\psfrag{a}{$\alpha_{t_i-\epsilon}$}
\psfrag{A}{$\alpha_{t_i+\epsilon}$}
\psfrag{b}{$\beta$}
\[\includegraphics[width=3.5 in]{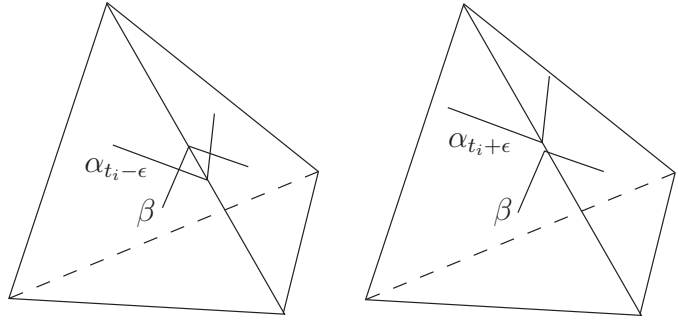}\]
\caption{Positive and negative intersections cancel as $t$ increases through $t_i$.}
\label{f:PositiveNegativeCancellation}
\end{figure}

Just before (or after) $t_i$, $\alpha_t$ meets $\beta$ as in Figure \ref{f:PositiveNegativeCancellation}. Here we see two intersections, one of each normal sign, of $\alpha_t \cap \beta$ which cancel as $t$ increases through $t_i$. It follows that $\eta_\sigma(\alpha_{t_i-\epsilon} \cap \beta)=\eta_\sigma(\alpha_{t_i+\epsilon} \cap \beta)$. As $\alpha' \cap \beta=\emptyset$, we conclude $\eta_\sigma(\alpha_{t} \cap \beta)$ is zero for all non-critical $t$. In particular, it must have been the case that $\eta_\sigma(\alpha \cap \beta)=0$

Let $\eta_\sigma(A \cap B)$ now denote the sum, over all curves $\alpha$ of $A \cap \bdy \sigma$ and $\beta$ of $B \cap \bdy \sigma$ of $\eta_{\sigma}(\alpha \cap \beta)$. It follows from the above argument that $\eta_\sigma(A \cap B)=0$. Thus, the sum over all tetrahedra $\sigma$ not in $\Delta$ of $\eta_\sigma(A \cap B)$ is also zero. 

Now note that if $\delta$ is an interior 2-simplex, then the normal sign of any intersection point of $A \cap \delta$ and $B \cap \delta$ is opposite from the perspective of the tetrahedra on either side of $\delta$. Hence, the sum of $\eta_\sigma(A \cap B)$, over all tetrahedra $\sigma$, must be equal to the difference of the number of positive and negative intersection points of $\bdy _\Delta A \cap \bdy _\Delta B$. As we have reasoned above that this total is zero, the result follows. 
\end{proof}

\section{Main Proof}

We now put the three notions of axis, twisting, and compatibility together in the following definition. 

\begin{dfn}
\label{d:consistency}
Let $M$ be a closed, oriented, triangulated 3-manifold, and let $\Delta$ be a set of tetrahedra in the triangulation of $M$. Two locally helical surfaces $H$ and $G$ are said to be {\it consistent} in $(M; \Delta)$ if they are compatible outside of $\Delta$, and if for all $\sigma \in \Delta$, $H \cap \sigma$ and $G \cap \sigma$ have the same handedness with respect to the same choice of axis.
\end{dfn}

Theorem \ref{t:mainthm} is a consequence of the following lemma. 

\begin{lem}
\label{mainlem}
Let $M$ be a closed, oriented, triangulated 3-manifold, and let $\Delta$ be a set of tetrahedra in the triangulation of $M$. If $H$ and $G$ are consistent, locally helical surfaces in $(M;\Delta)$,  then the net twisting of $H \cap \Delta$ is the same as the net twisting of $G \cap \Delta$. 
\end{lem}

\begin{proof}
For each $\sigma \in \Delta$ let $H_\sigma=H \cap \sigma$ and $G_\sigma=G \cap \sigma$. As noted in the proof of Lemma \ref{l:NumTurns}, for each $\sigma \in \Delta$, $\bdy H_\sigma$ and $\bdy G_\sigma$ must be normal loops of the same type. Thus, by Lemma  \ref{l:SameSign}, for each $\sigma \in \Delta$, all points of $\bdy H_\sigma \cap \bdy G_\sigma$ have the same sign. Let $\Delta_+$ be the subset of $\Delta$ where this sign is positive, and $\Delta_-$ the subset of $\Delta$ where it is negative.  Thus, on each $\sigma \in \Delta_+$, 
\[\eta_\sigma (\bdy H_\sigma \cap \bdy G_\sigma)=|\bdy H_\sigma \cap \bdy G_\sigma|\]
and for all $\sigma \in \Delta_-$,
\[\eta_\sigma (\bdy H_\sigma \cap \bdy G_\sigma)=-|\bdy H_\sigma \cap \bdy G_\sigma|.\]

Consider the sum $\sum \limits _{\sigma \in \Delta}  \#(\bdy H_\sigma \cap \bdy G_\sigma)$, where $\#(\bdy H_\sigma \cap \bdy G_\sigma)$ denotes the signed intersection number of $\bdy H_\sigma$  and $\bdy G_\sigma$. Suppose $\sigma_1$ and $\sigma_2$ are adjacent tetrahedra in $\Delta$,   $p_1 \in \bdy H_{\sigma_1} \cap \bdy G_{\sigma_1}$, $p_2 \in \bdy H_{\sigma_2} \cap \bdy G_{\sigma_2}$, and $p_1$ is identified with $p_2$ in $M$.\footnote{Here we are allowing $\sigma_1$ to be equal to $\sigma_2$ when there are self-identifications, but in this case $p_1$ must be distinct from $p_2$.} Then the sign of $p_1$ will be opposite the sign of $p_2$, and thus $p_1$ and $p_2$ will cancel in  $\sum \limits _{\sigma \in \Delta}  \#(\bdy H_\sigma \cap \bdy G_\sigma)$. If, on the other hand, $p$ is a point of $\bdy H_{\sigma} \cap \bdy G_{\sigma}$ that is on a unique $\sigma \in \Delta$, then $p \in \bdy (H-\Delta) \cap \bdy(G-\Delta)$. By hypothesis, $H-\Delta$ and $G-\Delta$ are compatible surfaces, thus by Lemma \ref{l:NetZero} the number of positive and negative points of $\bdy (H-\Delta) \cap \bdy(G-\Delta)$ are equal. We conclude that $\sum \limits _{\sigma \in \Delta}  \#(\bdy H_\sigma \cap \bdy G_\sigma)=0$, or equivalently, 

\[\sum \limits _{\sigma \in \Delta_+}  |\bdy H_\sigma \cap \bdy G_\sigma| = \sum \limits _{\sigma \in \Delta_-}  |\bdy H_\sigma \cap \bdy G_\sigma|\]
and thus, 
\[\sum \limits _{\sigma \in \Delta_+}  \eta_\sigma(\bdy H_\sigma \cap \bdy G_\sigma) = -\sum \limits _{\sigma \in \Delta_-}  \eta_\sigma(\bdy H_\sigma \cap \bdy G_\sigma).\]

Applying Lemma \ref{l:NumTurns} to this equality now yields
\[\sum \limits _{\sigma \in \Delta_+} 2(t(H_\sigma)-t(G_\sigma))  = -\sum \limits _{\sigma \in \Delta_-}  2(t(H_\sigma)-t(G_\sigma)),\]
which implies
\begin{eqnarray*} 
0 & =& \sum \limits _{\sigma \in \Delta_+} 2(t(H_\sigma)-t(G_\sigma))+\sum \limits _{\sigma \in \Delta_-}  2(t(H_\sigma)-t(G_\sigma))\\
&=&\sum \limits _{\sigma \in \Delta}  2(t(H_\sigma)-t(G_\sigma)) \\
&=&\sum \limits _{\sigma \in \Delta}  t(H_\sigma)- \sum \limits _{\sigma \in \Delta} t(G_\sigma).
\end{eqnarray*}

Therefore, $\sum \limits _{\sigma \in \Delta}  t(H_\sigma) = \sum \limits _{\sigma \in \Delta} t(G_\sigma)$, i.e.~the net twisting is the same for all surfaces in the chosen consistency class.
\end{proof}

We are now ready to prove Theorem \ref{t:mainthm}.

\begin{proof}
Let $n$ be a positive integer, and consider the set of all locally helical surfaces (up to normal isotopy) that have total absolute twisting $\leq n$ in $M - \Delta$. The number of compatibility classes of surfaces in $M - \Delta$ is finite, since there are only a finite number of normal loops on each tetrahedron of length $\leq 4(n+1)$. Moreover, there are only three possible axes for each tetrahedron in $\Delta$, and two choices of handedness for each. Thus, the number of consistency classes for $(M; \Delta)$ is finite. Theorem \ref{t:mainthm} thus immediately follows from Lemma \ref{mainlem}.
\end{proof}

\bibliographystyle{alpha}

\end{document}